\documentclass[12pt, a4paper]{amsart}

\theoremstyle{plain}
\newtheorem{theorem}{Theorem}
\newtheorem{corollary}[theorem]{Corollary}
\newtheorem{proposition}{Proposition}
\newtheorem{lemma}[proposition]{Lemma}

\theoremstyle{definition}

\theoremstyle{remark}
\newtheorem{remark}[proposition]{Remark}

\numberwithin{equation}{section}

\newcommand{\R}{\mathbb R}

\newcommand{\ve}{\varepsilon}

\providecommand{\norm}[1]{\lVert#1\rVert}

\DeclareMathOperator{\Div}{div}

\title[the Navier-Stokes equations with the Coriolis force]{A remark on a priori estimate for the Navier-Stokes equations with the Coriolis force}

\author{Hiroki Ito}
\author[J. Kato]{Jun Kato}

\address{Graduate School of Mathematics, Nagoya University, Nagoya 464-8602, Japan}

\email[J. Kato]{jkato@math.nagoya-u.ac.jp}

\begin{document}

\begin{abstract}
The Cauchy problem for the Navier-Stokes equations with the Coriolis force is considered.
It is proved that a similar a priori estimate, which is derived for the Navier-Stokes equations by Lei and Lin \cite{Lei-Lin_2011}, holds under the effect of the Coriolis force.
As an application existence of a unique global solution for arbitrary speed of rotation is proved, as well as its asymptotic behavior.
\end{abstract}

\maketitle

\section{Introduction}
In this note, we consider the initial value problem of the Navier-Stokes equations with the Coriolis force in $\R^{3}$,
\begin{equation}\tag{$\mathrm{NS}_{\Omega}$}\label{NSC}
\left\{
\begin{aligned}
&\partial_{t}u -\nu \Delta u + \Omega e_{3}\times u + (u,\nabla)u+\nabla p = 0, & & \mbox{in}\ (0,\infty)\times \R^{3},\\
&\mathrm{div}\, u=0,  & &\mbox{in}\ (0,\infty)\times \R^{3},\\
&u|_{t=0}=u_{0}, & & \mbox{in}\ \R^{3},
\end{aligned}
\right.
\end{equation}
where $u=u(t,x)=(u_{1}(t,x),u_{2}(t,x),u_{3}(t,x))$ denotes the unknown velocity field, and $p=p(t,x)$ denotes the unknown scalar pressure, while $u_{0}=u_{0}(x)=(u_{0}^{1}(x),u_{0}^{2}(x),u_{0}^{3}(x))$ denotes the initial velocity field.
The constant $\nu >0$ denotes the viscosity coefficient of the fluid, and $\Omega\in\R$ represents the speed of rotation around the vertical unit vector $e_{3}=(0,0,1)$, which is called the Coriolis parameter.

Recently, this problem gained some attention due to its importance in applications to geophysical flows, see e.g.\ \cite{Majda_2003, CDGG_2006}.
Mathematically, $(\mathrm{NS}_{\Omega})$ also have a interesting feature that there exists a global solution for arbitrary large data provided the speed of rotation $\Omega$ is large enough, see e.g. \cite{BMN_1997, CDGG_2006, Iwabuchi-Takada_2013}.
There are another type of results which shows the existence of a global solution uniformly in $\Omega$ provided the data is sufficiently small, see e.g. \cite{Giga-Inui-Mahalov-Saal_2008, Hieber-Shibata_2010, Konieczny-Yoneda_2011, Iwabuchi-Takada_2014}.
The purpose of this note is, concerning to the latter, to relax the smallness condition of the data, based on the idea for the Navier-Stokes equations, $\Omega =0$ in \eqref{NSC}, by \cite{Lei-Lin_2011}.

Before stating our main results, we give a definition of function spaces.
For $m\in \R$, we define
$$
  \chi^{m}(\R^{3}):=\bigl\{f\in \mathcal{S}'\,|\, \widehat{f}\in L^{1}_{\mathrm{loc}},\
    \|f\|_{\chi^{m}}:=\int_{\R^{3}} |\xi|^{m} |\widehat{f}(\xi)|\,d\xi
    <\infty\bigr\}.
$$
In particular, we only use spaces $\chi^{-1}$, $\chi^{0}$, and $\chi^{1}$ below, so we summarize elementary estimates concerning the spaces we will use later.

\begin{lemma}\label{lemma-1}
$(1)$ For $s>1/2$, $\norm{f}_{\chi^{-1}(\R^{3})} \le C\norm{f}_{L^{2}}^{1-\frac{1}{2s}}\norm{f}_{\dot{H}^{s}}^{\frac{1}{2s}}$.

\vspace{6pt}
\noindent
$(2)$ $\norm{f}_{\chi^{0}}\le \norm{f}_{\chi^{-1}}^{1/2} \norm{f}_{\chi^{1}}^{1/2}$.

\vspace{6pt}
\noindent
$(3)$ $\norm{\nabla f}_{L^{\infty}} \le \norm{f}_{\chi^{1}}$.
\end{lemma}

\begin{proof}
(1) We take $R>0$, which is determined later, to divide the integral
\begin{align*}
  \|f\|_{\chi^{-1}}& =\int_{|\xi|\le R}|\xi|^{-1}|\widehat{f}(\xi)|\,d\xi
    + \int_{|\xi|> R}|\xi|^{-1}|\widehat{f}(\xi)|\,d\xi\\
  & \le \Bigl(\int_{|\xi|\le R}|\xi|^{-2}\,d\xi\Bigr)^{1/2} \|f\|_{L^{2}}
    + \Bigl(\int_{|\xi|> R}|\xi|^{-2-2s}\Bigr)^{1/2} \|f\|_{\dot{H}^{s}}\\
  & = |S^{2}|^{1/2} \Bigl(R^{1/2}\|f\|_{L^{2}} 
    + \frac{1}{\sqrt{2s-1}}R^{-s+1/2}\|f\|_{\dot{H}^{s}}\Bigr).
\end{align*}
Then, choosing $R=\|f\|_{L^{2}}^{-1/s}\|f\|_{\dot{H}^{2}}^{1/s}$, we obtain the desired result.

\noindent
(2) This estimate is easily derived by the H\"older inequality,
\begin{equation*}
  \|f\|_{\chi^{0}} 
  = \int |\xi|^{-1/2}|\widehat{f}(\xi)|^{1/2}|\xi|^{1/2}|\widehat{f}(\xi)|^{1/2}\,d\xi
  \le \norm{f}_{\chi^{-1}}^{1/2} \norm{f}_{\chi^{1}}^{1/2}.
\end{equation*}

\noindent
(3) This is also easily derived from the Fourier inversion formula and the Hausdorff-Young inequality.
\end{proof}

Now we state our main results.

\begin{theorem}\label{prop-apriori}
Let $\Omega\in\R$, and let $u_{0}\in \chi^{-1}$ satisfy $\mathrm{div}\,u_{0}=0$ and $\|u_{0}\|_{\chi^{-1}}<(2\pi)^{3}\nu$.
For $T>0$, assume that $u\in C([0,T);\,\chi^{-1})$ is a solution to $(\mathrm{NS}_{\Omega})$ in the distribution sense satisfying
\begin{equation*}
  u\in L^{1}(0,T;\,\chi^{1}),\quad \partial_{t}u\in L^{1}(0,T;\,\chi^{-1}).
\end{equation*}
Then, $u$ satisfies
\begin{equation}\label{apriori}
  \|u(t)\|_{\chi^{-1}} 
  + (\nu - (2\pi)^{-3}\|u_{0}\|_{\chi^{-1}})
  \int_{0}^{t}\|u(\tau)\|_{\chi^{1}}\,d\tau
  \le \| u_{0}\|_{\chi^{-1}},\quad 0\le t < T.
\end{equation}
\end{theorem}

\begin{remark}
(1) This a priori estimate is first derived in the case $\Omega =0$ in \cite[Proof of Theorem 1.1]{Lei-Lin_2011}. 
Here, Theorem \ref{prop-apriori} states that the same estimate also holds under the effect of the Coriolis force.

\vspace{6pt}
\noindent
(2) In this note, we define the Fourier transform of $f$ by
$$
  \widehat{f}(\xi)=\mathcal{F}[f](\xi):=\int e^{-i x\cdot \xi}f(x)\,dx.
$$
The constant $(2\pi)^{3}$ in the theorem appears from the following formula:
$$
  \quad \mathcal{F}[fg](\xi)=(2\pi)^{-3}(\widehat{f}\ast\widehat{g})(\xi),
$$
where $f\ast g$ denotes the convolution of $f$ and $g$.

\vspace{6pt}
\noindent
(3) From the a priori estimate \eqref{apriori}, we especially obtain
$$
  \|u\|_{L^{\infty}(0,T;\,\chi^{-1})}\le \|u_{0}\|_{\chi^{-1}},\quad
  \|u\|_{L^{1}(0,T;\, \chi^{1})} \le \frac{\|u_{0}\|_{\chi^{-1}}}{\nu - (2\pi)^{-3}\|u_{0}\|_{\chi^{-1}}}.
$$
\end{remark}

As an application of Theorem \ref{prop-apriori} we obtain a unique global solution to $(\mathrm{NS}_{\Omega})$.

\begin{theorem}\label{thm-0}
Let $\Omega\in\R$.
Assume that $u_{0}\in \chi^{-1}(\R^{3})$ satisfy $\mathrm{div}\,u_{0}=0$ and $\|u_{0}\|_{\chi^{-1}}<(2\pi)^{3}\nu$.
Then, there exists a unique global solution $u\in C([0,\infty);\,\chi^{-1})$ to $(\mathrm{NS}_{\Omega})$ satisfying
$$
  u\in L^{1}(0,\infty;\chi^{-1}),\quad \partial_{t}u\in L^{1}_{\mathrm{loc}}(0,\infty;\,\chi^{-1}),
$$
and
$$
  \sup_{t>0}\bigl\{\|u(t)\|_{\chi^{-1}}
  +(\nu -(2\pi)^{-3} \|u_{0}\|_{\chi^{-1}})\int_{0}^{t}\|u(\tau)\|_{\chi^{1}}\,d\tau \bigr\}
  \le \|u_{0}\|_{\chi^{^{-1}}}.
$$
\end{theorem}

\begin{remark}

\vspace{6pt}
\noindent
(1) There are several results which treats the existence of a unique global solution to \eqref{NSC}, see \cite{Iwabuchi-Takada_2014} and reference therein.
In particular, the spaces $FM_{0}^{-1}$, which is considered by Giga, Inui, Mahalov, and Saal \cite{Giga-Inui-Mahalov-Saal_2008}, and $\mathcal{B}^{-1}_{1,2}$ by \cite{Iwabuchi-Takada_2014}, are larger than $\chi^{-1}$.
However, the advantage of this result is that the condition of the size of the data is merely $\|u_{0}\|_{\chi^{-1}}<(2\pi)^{3}\nu$.

\vspace{6pt}
\noindent
(2)
In the Navier-Stokes equations, the case $\Omega =0$, the corresponding result is proved in \cite[Theorem 1.1]{Lei-Lin_2011}. 
We notice that there is also the another approach by \cite[Theorem 1.3]{Zhang-Yin_2013-P}.
In our forthcoming paper we will consider that approach for \eqref{NSC}.
\end{remark}

As a byproduct, we also obtain the following.

\begin{theorem}\label{thm1}
Let $s>3/2$ and $\Omega\in\R$.
Assume that $u_{0}\in H^{s}(\R^{3})$ satisfy $\mathrm{div}\,u_{0}=0$ and $\|u_{0}\|_{\chi^{-1}}<(2\pi)^{3}\nu$.
Then, there exists a unique global solution $u\in C([0,\infty);\,H^{s})$ to $(\mathrm{NS}_{\Omega})$ satisfying
$$
  u\in AC([0,\infty);\,H^{s-1})
  \cap L^{1}_{\mathrm{loc}}(0,\infty;H^{s+1})
$$
and
$$
  \sup_{t>0}\bigl\{\|u(t)\|_{\chi^{-1}}
  +(\nu -(2\pi)^{-3} \|u_{0}\|_{\chi^{-1}})\int_{0}^{t}\|u(\tau)\|_{\chi^{1}}\,d\tau \bigr\}
  \le \|u_{0}\|_{\chi^{^{-1}}}.
$$
\end{theorem}

\begin{remark}
Since $s>3/2$, we have $H^{s}\hookrightarrow \chi^{-1}$ by Lemma \ref{lemma-1}.
The condition $s>3/2$ follows from the local well-posedness by Proposition \ref{prop-lwp} which we employ for the proof.
For a interval $I$ and a Banach space $X$, $AC(I;\,X)$ denotes the space of $X$-valued absolutely continuous functions.
\end{remark}

Next theorem states the asymptotic behavior of a given global solution to \eqref{NSC} in the framework of Sobolev spaces.
\begin{theorem}\label{thm2}
Let $s>1/2$ and $\Omega\in\R$.
Assume that $u\in C([0,\infty);\,H^{s}(\R^{3}))$ is a global solution to $(\mathrm{NS}_{\Omega})$ satisfying
\begin{equation}\label{sol-cond}
  u\in AC([0,\infty);\,H^{s-1}(\R^{3}))
  \cap L^{1}_{\mathrm{loc}}(0,\infty;H^{s+1}(\R^{3})).
\end{equation}
Then, $\lim_{t\to\infty}\|u(t)\|_{\chi^{-1}}=0$.
\end{theorem}

\begin{remark}
In the Navier-Stokes case $\Omega =0$, this result corresponds to the result in \cite{Benameur_2015}.
In that result, the assumption is only $u\in C([0,\infty);\chi^{-1})$ is a global solution.
Compared with that result, additional assumptions \eqref{sol-cond} are imposed for the uniqueness of solutions.
\end{remark}

As an application of Theorem \ref{thm2} we obtain the following.

\begin{corollary}
The global solution to $(\mathrm{NS}_{\Omega})$ derived in Theorem \ref{thm1} satisfies
$$
  \lim_{t\to 0} \|u(t)\|_{\chi^{-1}}=0.
$$
\end{corollary}

\vspace{6pt}

This paper is organized as follows.
In Section 2 we give a proof of Theorem \ref{prop-apriori}.
In Section 3 we prove Theorem \ref{thm1} as an application of Theorem \ref{prop-apriori}.
In Section 4 we give a proof of Theorem \ref{thm-0} by using Theorem \ref{prop-apriori} and Theorem \ref{thm1}.
In Section 5 we give a proof of Theorem \ref{thm2}.

\vspace{6pt}

\section{Proof of Theorem \ref{prop-apriori}}

In this section we give a proof of Theorem \ref{prop-apriori}.

\begin{proof}[Proof of Theorem \ref{prop-apriori}]
By applying the Fourier transform to the equation, we have
$$
  \partial_{t}\widehat{u} + \nu |\xi|^{2}\widehat{u}
  + \Omega e_{3}\times \widehat{u}
  + \mathcal{F}\bigl[ (u,\nabla)u\bigr]
  + i\xi \widehat{p} =0.
$$
Thus, we obtain
\begin{align*}
\partial_{t}|\widehat{u}|^{2}
& = 2\mathrm{Re}(\partial_{t}\widehat{u}\cdot \overline{\widehat{u}})\\
& = -2\nu |\xi|^{2} |\widehat{u}|^{2}
  -2\Omega\, \mathrm{Re}\bigl[(e_{3}\times \widehat{u})\cdot\overline{\widehat{u}}\bigr]
  -2\mathrm{Re}\bigl\{\mathcal{F}\bigl[ (u,\nabla)u\bigr]\cdot\overline{\widehat{u}}\bigr\}
  -2 \mathrm{Re}\bigl[(i\xi \widehat{p})\cdot\overline{\widehat{u}}\bigr].
\end{align*}
Here, since 
$$
  (e_{3}\times \widehat{u})\cdot\overline{\widehat{u}}
  = -\widehat{u}_{2}\overline{\widehat{u}_{1}} + \widehat{u}_{1}\overline{\widehat{u}_{2}}
  = 2i\, \mathrm{Im}\bigl[ \widehat{u}_{1}\overline{\widehat{u}_{2}}\bigr],
$$
we observe that $\mathrm{Re}[(e_{3}\times \widehat{u})\cdot\overline{\widehat{u}}]=0$.
Also, we have $(i\xi \widehat{p})\cdot\overline{\widehat{u}}=0$, since $\mathrm{div}\,u=0$. Moreover, we notice that
\begin{align*}
  \mathcal{F}\bigl[ (u,\nabla)u\bigr]_{j}(\xi)
  & = \sum_{k=1}^{3} (2\pi)^{-3}\widehat{u}_{k}\ast\widehat{\partial_{k}u}_{j}(\xi)\\
  & = \sum_{k=1}^{3} (2\pi)^{-3} \int \widehat{u}_{k}(\xi-\eta)\, i\eta_{k}\widehat{u}_{j}(\eta)\,d\eta\\
  & = \sum_{k=1}^{3} (2\pi)^{-3} i\xi_{k}\int\widehat{u}_{k}(\xi-\eta)\, \widehat{u}_{j}(\eta)\,d\eta,
\end{align*}
since $\sum_{k=1}^{3}(\xi_{k}-\eta_{k})\widehat{u}_{k}(\xi-\eta)=0$.
Therefore, we obtain
\begin{align*}
  \partial_{t}|\widehat{u}|^{2} + 2\nu |\xi|^{2} |\widehat{u}|^{2}
  & \le 2(2\pi)^{-3} \sum_{j,k=1}^{3}|\xi_{k}|\,(|\widehat{u}_{k}|\ast|\widehat{u}_{j}|)\,|{u}_{j}|\\
  & \le 2(2\pi)^{-3} |\xi|\,|\widehat{u}|\,(|\widehat{u}|\ast|\widehat{u}|).
\end{align*}
Then, for $\ve >0$, we observe that
\begin{align*}
  \partial_{t}(|\widehat{u}|^{2}+\ve)^{1/2}
  & = \frac{\partial_{t}|\widehat{u}|^{2}}{\ 2(|\widehat{u}|^{2}+\ve)^{1/2}}\\
  & \le -\frac{\nu |\xi|^{2}|\widehat{u}|^{2}}{\ (|\widehat{u}|^{2}+\ve)^{1/2}}
  + (2\pi)^{-3} \frac{|\xi|\,|\widehat{u}|}{\ (|\widehat{u}|^{2}+\ve)^{1/2}}\,(|\widehat{u}|\ast|\widehat{u}|).
\end{align*}
Integrating with respect to $t$, we obtain
\begin{align*}
  & (|\widehat{u}(t,\xi)|^{2}+\ve)^{1/2}
  + \int_{0}^{t} \frac{\nu |\xi|^{2}|\widehat{u}(\tau,\xi)|^{2}}{(|\widehat{u}(\tau,\xi)|^{2}+\ve)^{1/2}}\,d\tau\\
  & \le (|\widehat{u}_{0}(\xi)|^{2}+\ve)^{1/2}
  + (2\pi)^{-3} \int_{0}^{t}\frac{|\xi|\,|\widehat{u}(\tau,\xi)|}{\ (|\widehat{u}(\tau,\xi)|^{2}+\ve)^{1/2}}\,(|\widehat{u}(\tau)|\ast|\widehat{u}(\tau)|)(\xi)\,d\tau.
\end{align*}
Then, letting $\ve\to 0$, we get
$$
  |\widehat{u}(t,\xi)|
  + \int_{0}^{t} \nu |\xi|^{2}|\widehat{u}(\tau,\xi)|\,d\tau
  \le |\widehat{u}_{0}(\xi)|
  + (2\pi)^{-3} \int_{0}^{t}|\xi|\,(|\widehat{u}(\tau)|\ast|\widehat{u}(\tau)|)(\xi)\,d\tau.
$$
Finally, dividing by $|\xi|$, and then integrating over $\R^{n}$, we obtain
$$
  \| u(t)\|_{\chi^{-1}} + \nu \int_{0}^{t} \| u(\tau)\|_{\chi^{1}}\,d\tau
  \le \|u_{0}\|_{\chi^{-1}} + (2\pi)^{-3} \int_{0}^{t} \| u(\tau)\|_{\chi^{0}}^{2}\,d\tau.
$$
By applying Lemma \ref{lemma-1} (2), we obtain,
\begin{equation}\label{ap-1}
  \| u(t)\|_{\chi^{-1}} + \nu  \| u\|_{L^{1}((0,t);\,\chi^{1})}
  \le \|u_{0}\|_{\chi^{-1}} + (2\pi)^{-3} \| u\|_{L^{\infty}((0,t);\,\chi^{-1})} \|u\|_{L^{1}((0,t);\,\chi^{1})}.
\end{equation}

To derive the desired estimate \eqref{apriori}, it suffices to prove that
$$
  \|u\|_{L^{\infty}((0,t);\, \chi^{-1})}\le \|u_{0}\|_{\chi^{-1}},\quad 0\le t < T.
$$
For the proof, we first show that
\begin{equation}\label{step1}
  \| u(t) \|_{\chi^{-1}} < (2\pi)^{3}\nu,\quad 0\le t<T
\end{equation}
holds by contradiction.
From the assumption $\|u_{0}\|_{\chi^{-1}}<(2\pi)^{3}\nu$ and $u\in C([0,T);\chi^{-1})$, we observe that there exists $\delta >0$ such that \eqref{step1} holds on $[0,\delta)$.
Now assume that there exists $t_{0}\in (0,T)$ such that $\| u(t) \|_{\chi^{-1}} < (2\pi)^{3}\nu$ for $0<t<t_{0}$ and
$$
  \| u(t_{0}) \|_{\chi^{-1}} = (2\pi)^{3}\nu,
$$
then by \eqref{ap-1} we reach the contradiction
$$
  (2\pi)^{3}\nu = \| u(t_{0}) \|_{\chi^{-1}} \le \| u_{0} \|_{\chi^{-1}} 
  < (2\pi)^{3}\nu,
$$
since $\|u\|_{L^{\infty}((0,t_{0});\, \chi^{-1})}=(2\pi)^{3}\nu$.
Therefore, we obtain \eqref{step1}.
Finally, applying \eqref{step1} to estimate on the right hand side of \eqref{ap-1}, we obtain
$$
  \| u(t) \|_{\chi^{-1}}< \|u_{0}\|_{\chi^{-1}},\quad 0\le t< T.
$$
This completes the proof.
\end{proof}

\vspace{0cm}

\section{Proof of Theorem \ref{thm1}}

Below we fix $\Omega\in\R$.
For the existence of local solutions, we employ the following result.

\begin{proposition}\label{prop-lwp}
Let $s>3/2$.
For $u_{0}\in H^{s}(\R^{3})$ with $\mathrm{div}\,u_{0}=0$, there exists $T=T(|\Omega|, s, \|u_{0}\|_{H^{s}})>0$ such that $(\mathrm{NS}_{\Omega})$ admits a unique strong solution $u\in C([0,T];\,H^{s}(\R^{3}))$ satisfying 
$$
  u\in AC([0,T];\,H^{s-1}(\R^{3}))
  \cap L^{1}(0,T;\,H^{s+1}(\R^{3})).
$$
\end{proposition}

\begin{remark}
(1) For the proof, we refer to \cite[Lemma 3.1]{Koh-Lee-Takada_2014}.
The idea is based on to construct the solution to the integral equation
\begin{equation*}
  u(t)=e^{\nu t\Delta}u_{0}
  -\Omega \int_{0}^{t}e^{\nu (t-\tau)\Delta}\mathbb{P}(e_{3}\times u)(\tau)\,d\tau
  -\int_{0}^{t}e^{\nu (t-\tau)\Delta}\mathbb{P}(u, \nabla u)(\tau)\,d\tau
\end{equation*}
by the contraction mapping argument, where $\mathbb{P}=(\delta_{ij}+R_{i}R_{j})_{i,j}$ is the Helmholtz projection.
We notice that the condition in \cite[Lemma 3.1]{Koh-Lee-Takada_2014} is $s>3/2+1$, because their main subject is the Euler equation.
For the above statement, $s>3/2$ is sufficient.

\vspace{6pt}
\noindent
(2) In this proposition, the size of $T$ is characterized by
\begin{equation}\label{cond_T}
  C_{0}|\Omega|T + C_{1}\|u_{0}\|_{H^{s}}(T+T^{1/2}\nu^{-1/2})\le \frac{1}{2}.
\end{equation}

\vspace{6pt}
\noindent
(3) Since $s>3/2$, the solution constructed by Proposition \ref{prop-lwp} satisfies the assumptions in Theorem \ref{prop-apriori}.
In particular, since 
\begin{equation*}
  \partial_{t}u=\nu\Delta u - \Omega \mathbb{P}(e_{3}\times u)  - \mathbb{P}(u,\nabla u)\quad \mbox{in}\ H^{s-1}
\end{equation*}
holds for a.e.\ $t\in (0,T)$, we easily observe that $\partial_{t}u\in L^{1}(0,T;\,\chi^{-1})$.
\end{remark}

We will use the following energy estimate.
\begin{proposition}\label{prop-energy}
Let $s\ge 0$ and $T>0$.
Assume that $u\in C([0,T);\,H^{s}(\R^{3}))$ is a solution to $(\mathrm{NS}_{\Omega})$ satisfying
$$
  u\in AC([0,T);\,H^{s-1}(\R^{3}))
  \cap L^{1}(0,T;H^{s+1}(\R^{3})).
$$
Then, $u$ satisfies
$$
  \|u(t)\|_{H^{s}} 
    \le \| u_{0}\|_{H^{s}}e^{C\int_{0}^{T}\|\nabla u(\tau)\|_{L^{\infty}}\,d\tau},
    \quad 0\le t < T.
$$
\end{proposition}

\begin{remark}
For the proof of this proposition, we also refer to \cite[Proof of Theorem 4.1]{Koh-Lee-Takada_2014}.
There, we easily observe that
$$
  \frac{d}{dt}\|u(t)\|_{H^{s}}
  \le C\| \nabla u(t)\|_{L^{\infty}} \|u(t)\|_{H^{s}}
$$
holds for $s\ge 0$.
We notice that the term concerning $\Omega e_{3}\times u$ vanishes due to 
\begin{equation*}
  \Omega (e_{3}\times u)\cdot u =0.
\end{equation*}
\end{remark}


\vspace{10pt}
Now we are in a position to prove Theorem \ref{thm1}.

\begin{proof}[Proof of Theorem \ref{thm1}.]
Let $T^{\ast}$ be the maximal existence time of a unique solution derived by applying Proposition \ref{prop-lwp} repeatedly.
Now assume $T^{\ast}<\infty$.
Then, by \eqref{cond_T}, we must have
\begin{equation}\label{thm1-1}
  \lim_{t\to T^{\ast}} \| u(t)\|_{H^{s}}=\infty.
\end{equation}

Since this solution satisfies the energy estimate  in Proposition \ref{prop-energy}, we have
$$
  \|u(t)\|_{H^{s}} 
    \le \| u_{0}\|_{H^{s}}e^{C\int_{0}^{T^{\ast}}\|\nabla u(\tau)\|_{L^{\infty}}\,d\tau},
    \quad 0\le t < T^{\ast}.
$$
Then, since $\|u_{0}\|_{\chi^{-1}} < (2\pi)^{3}\nu$, applying Theorem \ref{prop-apriori} we obtain
$$
  \int_{0}^{T^{\ast}}\|\nabla u(\tau)\|_{L^{\infty}}\,d\tau
  \le \| u\|_{L^{1}(0,T^{\ast};\,\chi^{1})}
  \le \frac{\| u_{0}\|_{\chi^{-1}}}{\nu - (2\pi)^{-3}\| u_{0}\|_{\chi^{-1}}}.
$$
This implies $\sup_{0<t<T^{\ast}}\|u(t)\|_{H^{s}}<\infty$, which contradicts to \eqref{thm1-1}.
\end{proof}

\vspace{6pt}

\section{Proof of Theorem \ref{thm-0}}
In this section we give a proof of Theorem \ref{thm-0}.

For $u_{0}\in\chi^{-1}$ and $R>0$, we set
\begin{equation*}
  D_{R}=\{\xi\in \R^{3}\,|\, |\xi|\le R,\ |\widehat{u}_{0}(\xi)|\le R\},\quad
  u_{0}^{R}=\mathcal{F}^{1}[\chi_{D_{R}}\widehat{u_{0}}],
\end{equation*}
where $\chi_{D_{R}}$ denotes the characteristic function of $D_{R}$.
Then, we observe that
$$
  u_{0}^{R}\in H^{\infty},\quad \|u_{0}^{R}\|_{\chi^{-1}}\le \|u_{0}\|_{\chi^{-1}},
$$
and from Lebesgue's dominant convergence theorem,
\begin{equation}\label{data-conv}
  \|u_{0}^{R}-u_{0}\|_{\chi^{-1}}
  = \int |\xi|^{-1}\bigl(\chi_{D_{R}}(\xi)-1\bigr)
  |\widehat{u}_{0}(\xi)|\,d\xi \to 0,\quad R\to\infty,
\end{equation}
since $u_{0}\in\chi^{-1}$.

Now we apply Theorem \ref{thm1} for the data $u_{0}^{R}$ to derive a unique global solution $u^{R}\in C([0,\infty);\,H^{s})$ satisfying
$$
  u^{R}\in AC([0,\infty);\,H^{s-1})
  \cap L^{1}(0,\infty;H^{s+1})
$$
for $s>3/2$, and
\begin{equation}\label{ur-est}
  \| u^{R}\|_{L^{\infty}(0,\infty;\,\chi^{-1})}\le \|u_{0}\|_{\chi^{-1}},\quad
  \| u^{R}\|_{L^{1}(0,\infty;\,\chi^{1})}\le \frac{\|u_{0}\|_{\chi^{-1}}}{\nu - (2\pi)^{-3}\|u_{0}\|_{\chi^{-1}}}.
\end{equation}

Below we first show that $\{u^{R}\}$ is a Cauchy sequence in $L^{\infty}(0,\infty;\,\chi^{-1})$.
If we set $w=u^{R}-u^{R'}$, then $w$ satisfies
$$
  \partial_{t}w-\nu \Delta w + \Omega e_{3}\times w
  + (u^{R},\nabla)w + (w,\nabla)u^{R'} + \nabla(p^{R}-p^{R'})=0.
$$
Then, from the argument in the proof of Theorem \ref{prop-apriori}, we obtain
\begin{align*}
  & \|w(t)\|_{\chi^{-1}} + \nu \int_{0}^{t}\|w(\tau)\|_{\chi^{1}}\,d\tau\\
  & \le \|w(0)\|_{\chi^{-1}} + (2\pi)^{-3} \int_{0}^{t}\bigl(\|u^{R}(\tau)\|_{\chi^{0}}
    + \|u^{R'}(\tau)\|_{\chi^{0}}\bigr)\|w(\tau)\|_{\chi^{0}}\,d\tau.
\end{align*}
Here, applying Lemma \ref{lemma-1} (2) we have
\begin{equation}\label{chi0-est}
\begin{aligned}
  \|u^{R}\|_{\chi^{0}} \|w\|_{\chi^{0}}
  & \le \|u^{R}\|_{\chi^{-1}}^{1/2}\|u^{R}\|_{\chi^{1}}^{1/2}
    \|w\|_{\chi^{-1}}^{1/2}\|w\|_{\chi^{1}}^{1/2}\\
  & \le \frac{1}{2}(\|u^{R}\|_{\chi^{-1}}\|w\|_{\chi^{1}}
    + \|u^{R}\|_{\chi^{1}}\|w\|_{\chi^{-1}}).
\end{aligned}
\end{equation}
Therefore, combining \eqref{ur-est} we obtain
\begin{equation}\label{diff-est}
\begin{aligned}
  \|w(t)\|_{\chi^{-1}} + (\nu - (2\pi)^{-3} & \|u_{0}\|_{\chi^{-1}}) \int_{0}^{t}\|w(\tau)\|_{\chi^{1}}\,d\tau\\
  &\le \|w(0)\|_{\chi^{-1}} +  \int_{0}^{t} a(\tau)\,\|w(\tau)\|_{\chi^{-1}}\,d\tau,
\end{aligned}
\end{equation}
where
$$
  a(\tau)=\frac{1}{2(2\pi)^{3}}\bigl(\|u^{R}(\tau)\|_{\chi^{1}}+\|u^{R'}(\tau)\|_{\chi^{1}}\bigr).
$$
Note that by \eqref{ur-est} we have a uniform bound
$$
  \int_{0}^{\infty}a(\tau)\,d\tau \le \frac{\|u_{0}\|_{\chi^{-1}}}{(2\pi)^{3}\nu - \|u_{0}\|_{\chi^{-1}}}.
$$
Thus, applying Gronwall's inequality to \eqref{diff-est} we obtain
\begin{equation}\label{w-1-est}
  \| w(t)\|_{\chi^{-1}} \le \|w(0)\|_{\chi^{-1}}e^{\int_{0}^{t}a(\tau)\,d\tau},
\end{equation}
which implies 
$$
  \|u^{R}-u^{R'}\|_{L^{\infty}(0,\infty;\,\chi^{-1})}\le  \|u_{0}^{R}-u_{0}^{R'}\|_{\chi^{-1}}
  e^{\int_{0}^{\infty}a(\tau)\,d\tau} \to 0,\quad R,\, R'\to\infty.
$$
Therefore, there exists $u\in L^{\infty}(0,\infty;\,\chi^{-1})$ such that $u^{R}\to u$ in $L^{\infty}(0,\infty;\,\chi^{-1})$.

We next show the convergence in $L^{1}(0,\infty;\,\chi^{1})$.
The convergence in $L^{\infty}(0,\infty;\,\chi^{-1})$ implies there exists a subsequence $\{u^{\widetilde{R}}\}$ such that for a.e.\ $(t,\xi)$,
$$
  \mathcal{F}\bigl[u^{\widetilde{R}}\bigr](t,\xi) \to \widehat{u}(t,\xi),\quad R\to\infty.
$$
Therefore, by Fatou's lemma and the estimate derived from \eqref{diff-est} and \eqref{w-1-est},
$$
  \|w\|_{L^{1}(0,\infty;\,\chi^{1})} \le \frac{\|w(0)\|_{\chi^{-1}}}{\nu - (2\pi)^{-3} \|u_{0}\|_{\chi^{-1}}}\Bigl(1 + \int_{0}^{\infty}a(\tau)\,d\tau e^{\int_{0}^{\infty}a(\tau)\,d\tau} \Bigr),
$$
we conclude that
$$
  \|u^{R}-u\|_{L^{1}(0,\infty;\,\chi^{1})}
  \le \liminf_{\widetilde{R}\to 0} \|u^{R}-u^{\widetilde{R}}\|_{L^{1}(0,\infty;\,\chi^{1})}
  \to 0,\quad R\to\infty.
$$

From convergence in $L^{\infty}(0,\infty;\,\chi^{-1})\cap L^{1}(0,\infty;\,\chi^{1})$ we observe that the limit $u$ satisfies the integral equation
\begin{equation*}
  u(t)=e^{\nu t\Delta}u_{0}
  -\Omega \int_{0}^{t}e^{\nu (t-\tau)\Delta}\mathbb{P}(e_{3}\times u)(\tau)\,d\tau
  -\int_{0}^{t}e^{\nu (t-\tau)\Delta}\mathbb{P}\,\nabla\cdot(u\otimes u)(\tau)\,d\tau,
\end{equation*}
which $u^{R}$ also satisfies for the data $u_{0}^{R}$.
In fact, we are able to estimate
\begin{align*}
  & \|e^{\nu t \Delta}u_{0}^{R}-e^{\nu t \Delta}u_{0}\|_{\chi^{-1}}
  \le \|u_{0}^{R}-u_{0}\|_{\chi^{-1}},\\
  & \Bigl\| \int_{0}^{t}e^{\nu (t-\tau)\Delta}\mathbb{P}(e_{3}\times u^{R})(\tau)\,d\tau
    - \int_{0}^{t}e^{\nu (t-\tau)\Delta}\mathbb{P}(e_{3}\times u)(\tau)\,d\tau\Bigr\|_{\chi^{-1}}\\
    & \le t\|u^{R}-u\|_{L^{\infty}(0,\infty;\,\chi^{-1})},
\end{align*}
and
\begin{align*}
& \Bigl\|\int_{0}^{t}e^{\nu (t-\tau)\Delta}\mathbb{P}(u^{R}, \nabla u^{R})(\tau)\,d\tau
  - \int_{0}^{t}e^{\nu (t-\tau)\Delta}\mathbb{P}(u, \nabla u)(\tau)\,d\tau \Bigr\|_{\chi^{-1}}\\
  & \le \int_{0}^{t} \bigl(\|u^{R}(\tau)\|_{\chi^{0}}+\|u(\tau)\|_{\chi^{0}}\bigr)\|u^{R}(\tau)-u(\tau)\|_{\chi^{0}}\,d\tau\\
  & \le C\bigl( \|u^{R}-u\|_{L^{\infty}(0,\infty;\,\chi^{-1})} + 
  \|u^{R}-u\|_{L^{1}(0,\infty;\,\chi^{1})}\bigr),
\end{align*}
where we applied the estimate like \eqref{chi0-est} and the uniform bound \eqref{ur-est}.

We next show $\partial_{t} u\in L^{1}(0,T;,\chi^{-1})$ for any $T>0$, which implies $u\in C([0,\infty);\,\chi^{-1})$.
To prove this, we consider to apply $\partial_{t}$ to the right hand of the integral equation.
We first notice that for the first term
\begin{align*}
  \|\partial_{t}e^{\nu \Delta t}u_{0}\|_{L^{1}(0,\infty;\chi^{-1})}
  & = \|\nu \Delta e^{\nu \Delta t}u_{0}\|_{L^{1}(0,\infty;\chi^{-1})}\\
  & = \nu \int_{0}^{\infty} \int |\xi| e^{-\nu|\xi|^{2}t} |\widehat{u_{0}}(\xi)|\,d\xi dt\\
  & = \int |\xi|^{-1}|\widehat{u_{0}}(\xi)|\,d\xi
  = \|u_{0}\|_{\chi^{-1}}
\end{align*}
holds by changing the order of the integrals.
This type of argument can be found in \cite[Lemma 3.5]{Konieczny-Yoneda_2011}.
(See also \cite[Theorem 2.5]{Giga-Saal_2011} in relation with the $L^{1}$-maximal regularity.)
So, it suffices to show that $\partial_{t}\Phi\in L^{1}(0,T;\,\chi^{-1})$, where
$$
  \Phi(t)=\Omega \int_{0}^{t}e^{\nu (t-\tau)\Delta}\mathbb{P}(e_{3}\times u)(\tau)\,d\tau
  +\int_{0}^{t}e^{\nu (t-\tau)\Delta}\mathbb{P}\,\nabla\cdot(u\otimes u)(\tau)\,d\tau.
$$
Since
$$
  \partial_{t}\Phi (t)=\Delta \Phi(t) + \mathbb{P}(e_{3}\times u)(t) 
  + \mathbb{P}\,\nabla\cdot(u\otimes u)(t),
$$
we will check each term on the right hand side belongs to $L^{1}(0,T;\,\chi^{-1})$.
It is easy to see that
\begin{align*}
  & \int_{0}^{T}\| \mathbb{P}(e_{3}\times u) \|_{\chi^{-1}}\,dt
  \le T\|u\|_{L^{\infty}(0,T;\chi^{-1})},\\
  & \int_{0}^{T} \| \mathbb{P}\,\nabla\cdot(u\otimes u)(t)\|_{\chi^{-1}}\,dt
  \le \int_{0}^{T} \|u\|_{\chi^{0}}^{2}\,d\tau
  \le \|u\|_{L^{\infty}(0,T;\chi^{-1})}\|u\|_{L^{1}(0,T;\chi^{1})},\\
  & \int_{0}^{T}\Bigl\|\Omega \int_{0}^{t} \Delta e^{\nu (t-\tau)\Delta}\mathbb{P}(e_{3}\times u)(\tau)\,d\tau \Bigr\|_{\chi^{-1}}\,dt
  \le |\Omega|T \|u\|_{L^{1}(0,T;\,\chi^{1})}.
\end{align*}
And applying the argument the above again,
\begin{align*}
  & \int_{0}^{T}\Bigl\| \int_{0}^{t} \Delta e^{\nu (t-\tau)\Delta}\mathbb{P}\,\nabla\cdot(u\otimes u)(\tau)\,d\tau \Bigr\|_{\chi^{-1}}\,dt\\
  & \le \int_{0}^{T}\Bigl(\int_{0}^{t}\int |\xi|^{2}e^{-\nu (t-\tau) |\xi|^{2}} (|\widehat{u}(\tau)|\ast|\widehat{u}(\tau)|)(\xi)\,d\xi\, d\tau\Bigr)\, dt\\
  & = \int_{0}^{T} \int \Bigl(\int_{\tau}^{T}|\xi|^{2}e^{-\nu (t-\tau) |\xi|^{2}}\,dt\Bigr)
  (|\widehat{u}(\tau)|\ast|\widehat{u}(\tau)|)(\xi)\,d\xi\,d\tau\\
  & \le \int_{0}^{T} \|u\|_{\chi^{0}}^{2}\,d\tau
  \le \|u\|_{L^{\infty}(0,T;\chi^{-1})}\|u\|_{L^{1}(0,T;\chi^{1})}.
\end{align*}

Finally, we notice that \eqref{w-1-est} implies the uniqueness of solutions.

\vspace{6pt}

\section{Proof of Theorem \ref{thm2}}
In this section we give a proof of Theorem \ref{thm2}.

We take $\ve>0$ arbitrary small.
Since $u_{0}\in H^{s}\hookrightarrow\chi^{-1}$, we are able to choose $R_{0}>0$ such that
$$
  \int_{|\xi|> R_{0}} |\xi|^{-1}
  |\widehat{u}_{0}(\xi)|\,d\xi < \frac{\ve}{2}.
$$
Now we set
$$
  v_{0}=\mathcal{F}^{-1}[\chi_{\{|\xi|\le R_{0}\}}\widehat{u_{0}}],\quad
  w_{0}=\mathcal{F}^{-1}[\chi_{\{|\xi|> R_{0}\}}\widehat{u_{0}}].
$$
Then, we observe that $v_{0}\in H^{\infty}$, $w_{0}\in H^{s}$, $u_{0}=v_{0}+w_{0}$, and
$$
  \|w_{0}\|_{\chi^{-1}}<\frac{\ve}{2}.
$$

By applying Theorem \ref{thm1} for the initial data $w_{0}$ we obtain the solution $(w,p_{w})$ to $(\mathrm{NS}_{\Omega})$.
Then, $w\in C([0,\infty);\, H^{s})\cap L^{1}(0,\infty;\,H^{s+1})$ satisfies
\begin{equation}\label{w-est}
  \|w(t)\|_{\chi^{-1}} + (\nu-(2\pi)^{-3}\|w_{0}\|_{\chi^{-1}})\int_{0}^{t}\|w(\tau)\|_{\chi^{1}}\,d\tau \le \|w_{0}\|_{\chi^{-1}}<\frac{\ve}{2},\quad t>0.
\end{equation}

Now we set $v:=u-w$.
Then, $v\in C([0,\infty);\, H^{s})$ satisfies
\begin{equation*}
  v\in AC([0,\infty);\,H^{s-1}(\R^{3}))
  \cap L^{1}(0,\infty;H^{s+1}(\R^{3}))
\end{equation*}
and
\begin{equation*}
\left\{
\begin{aligned}
&\partial_{t}v+\nu \Delta v + \Omega e_{3}\times v + (v,\nabla)v + (w,\nabla)v + (v,\nabla)w +\nabla (p-p_{w}) = 0,\\ 
&\mathrm{div}\, v=0,\\  
&v|_{t=0}=v_{0}. 
\end{aligned}
\right.
\end{equation*}
Taking $L^{2}$-inner product with $v$, the equation becomes
\begin{equation*}
\frac{d}{dt}\|v(t)\|_{L^{2}}^{2} + \nu \|\nabla v(t)\|_{L^{2}}^{2}
= \langle (v,\nabla )w,v \rangle_{L^{2}}.
\end{equation*}
Since
$$
  \langle (v,\nabla )w,v \rangle_{L^{2}} = - \langle w, (v,\nabla)v\rangle_{L^{2}},
$$
we obtain
\begin{align*}
|\langle (v,\nabla )w,v \rangle_{L^{2}}|
& \le \| w \|_{L^{\infty}} \|v \|_{L^{2}} \| \nabla v\|_{L^{2}}\\
& \le C \|w\|_{\chi^{0}}\|v \|_{L^{2}} \| \nabla v\|_{L^{2}}\\
& \le C_{\nu} \|w\|_{\chi^{0}}^{2}\|v \|_{L^{2}}^{2}
  + \frac{\nu}{2}\, \| \nabla v\|_{L^{2}}^{2}
\end{align*}
Therefore, we obtain
\begin{equation*}
  \frac{d}{dt}\|v(t)\|_{L^{2}}^{2} + \frac{\nu}{2}\, \|\nabla v(t)\|_{L^{2}}^{2}
  = C_{\nu} \|w(t)\|_{\chi^{0}}^{2}\|v(t) \|_{L^{2}}^{2}.
\end{equation*}
Then, by Gronwall's inequality,
\begin{equation}\label{l2-energy}
  \|v(t)\|_{L^{2}}^{2} + \frac{\nu}{2}\, \int_{0}^{t}\|\nabla v(t)\|_{L^{2}}^{2}
  \le \| v(0)\|_{L^{2}}^{2} e^{C_{\nu} \int_{0}^{t} \|w(\tau)\|_{\chi^{0}}^{2}\,d\tau}.
\end{equation}
Here, by \eqref{w-est} we have
\begin{equation}\label{exp-est}
  \int_{0}^{t}\|w(\tau)\|_{\chi^{0}}^{2}\,d\tau
  \le \|w\|_{L^{\infty}((0,t);\,\chi^{-1})} \|w\|_{L^{1}((0,t);\,\chi^{1})}
  \le \frac{\| w_{0}\|_{\chi^{-1}}^{2}}{\nu - (2\pi)^{-3}\| w_{0}\|_{\chi^{-1}}}.
\end{equation}
Therefore, by Lemma \ref{lemma-1} (1), \eqref{l2-energy}, \eqref{exp-est} we obtain
\begin{equation*}
  \int_{0}^{\infty} \|v(t)\|_{\chi^{-1}}^{4}\,d\tau
  \le \int_{0}^{\infty} \| v(t)\|_{L^{2}}^{2} \|\nabla v(t)\|_{L^{2}}^{2}
  \le \frac{2}{\nu}\,\|v_{0}\|_{L^{2}}^{4}\,\exp\Bigl(\frac{C_{\nu}\,\| w_{0}\|_{\chi^{-1}}^{2}}{\nu - (2\pi)^{-3}\| w_{0}\|_{\chi^{-1}}}\Bigr).
\end{equation*}
Since $v\in C([0,\infty);\,\chi^{-1})$, we observe that there exists $t_{0}>0$ such that $\|v(t_{0})\|_{\chi^{-1}}<\ve/2$, and thus we have $\|u(t_{0})\|_{\chi^{-1}} \le \|v(t_{0})\|_{\chi^{-1}} + \|w(t_{0})\|_{\chi^{-1}} <\ve$.
So, applying Theorem \ref{thm1} for the data $u(t_{0})$ we obtain
$$
  \|u(t)\|_{\chi^{-1}}\le \|u(t_{0})\|_{\chi^{-1}}<\ve,\quad t>t_{0},
$$
which implies $\lim_{t\to 0} \|u(t)\|_{\chi^{-1}}=0$.

Here, we notice that in the final part of the proof we need the uniqueness of solutions, which is assured in our class of solutions.
In fact, if $u_{1}$, and $u_{2} \in C([0,\infty);\,H^{s})$ are two solutions to $(\mathrm{NS}_{\Omega})$ satisfying
$$
  u_{1},\ u_{2}\in AC([0,\infty);\,H^{s-1}(\R^{3}))
  \cap L^{1}_{\mathrm{loc}}(0,\infty;H^{s+1}(\R^{3})),
$$
then, $\widetilde{u}:=u_{1}-u_{2}$ satisfies $\Div\widetilde{u}=0$ and
$$
  \partial_{t}\widetilde{u}+\nu \Delta \widetilde{u} + \Omega e_{3}\times \widetilde{u} + (\widetilde{u},\nabla)\widetilde{u} + (u_{1},\nabla)\widetilde{u} + (\widetilde{u},\nabla)u_{2} +\nabla (p_{1}-p_{2}) = 0,
$$
and thus we obtain
\begin{equation*}
  \frac{d}{dt}\|\widetilde{u}(t)\|_{L^{2}}^{2} + \frac{\nu}{2}\, \|\nabla \widetilde{u}(t)\|_{L^{2}}^{2}
  = |\langle (\widetilde{u},\nabla )u_{2},\widetilde{u} \rangle_{L^{2}}|
  \le \|\nabla u_{2}(t)\|_{L^{\infty}}^{2}\|\widetilde{u}(t) \|_{L^{2}}^{2}.
\end{equation*}
Therefore, we have
\begin{equation*}
  \frac{d}{dt}\|\widetilde{u}(t)\|_{L^{2}}^{2} 
  = C \|u_{2}(t)\|_{H^{s+1}}\|\widetilde{u}(t) \|_{L^{2}}^{2}
\end{equation*}
and Gronwall's inequality implies $\widetilde{u}(t)=0$ for $t>0$.


\vspace{6pt}

\begin{thebibliography}{99}

\bibitem{BMN_1997}
A. Babin, A. Mahalov, B. Nicolaenko,
Regularity and integrability of 3D Euler and Navier-Stokes equations for rotating fluids,
Asymptot. Anal. \textbf{15} (1997), 103 -- 150.

\bibitem{Benameur_2015}
J. Benameur,
Long time decay to the Lei-Lin solution of 3D Navier-Stokes equations,
J. Math. Anal. Appl. \textbf{422} (2015), 424 -- 434.

\bibitem{CDGG_2006}
J.-Y. Chemin, B. Desjardins, I. Gallagher, E. Grenier,
``Mathematical geophysics. An introduction to rotating fluids and the Navier-Stokes equations.'' 
Oxford Lecture Series in Mathematics and its Applications \textbf{32},  Oxford University Press (2006).

\bibitem{Giga-Inui-Mahalov-Saal_2008}
Y. Giga, K. Inui, A. Mahalov, J. Saal,
Uniform global solvability of the rotating Navier-Stokes equations for nondecaying initial data,
Indiana Univ. Math. J. \textbf{57} (2008), 2775 -- 2791.

\bibitem{Giga-Saal_2011}
Y. Giga, J. Saal,
$L^{1}$ maximal regularity for the Laplacian and applications,
Discrete Contin. Dyn. Syst. 2011, Dynamical systems, differential equations and applications, 8th AIMS Conference, Suppl. Vol. I, 495 -- 504.


\bibitem{Hieber-Shibata_2010}
M. Hieber, Y. Shibata, 
The Fujita-Kato approach to the Navier-Stokes equations in the rotational framework,
Math. Z. \textbf{265} (2010), 481 -- 491.

\bibitem{Iwabuchi-Takada_2013}
T. Iwabuchi, R. Takada, 
Global solutions for the Navier-Stokes equations in the rotational framework,
Math. Ann. \textbf{357} (2013), 727 -- 741.

\bibitem{Iwabuchi-Takada_2014}
T. Iwabuchi, R. Takada, 
Global well-posedness and ill-posedness for the Navier-Stokes equations with the Coriolis force in function spaces of Besov type,
J. Funct. Anal. \textbf{267} (2014), 1321 -- 1337.

\bibitem{Koh-Lee-Takada_2014}
Y. Koh, S. Lee, R. Takada,
Strichartz estimates for the Euler equations in the rotational framework,
J. Differential Equations \textbf{256} (2014), 707 -- 744.

\bibitem{Konieczny-Yoneda_2011}
P. Konieczny, T. Yoneda,
On dispersive effect of the Coriolis force for the stationary Navier-Stokes equations,
J. Differential Equations \textbf{250} (2011), 3859 -- 3873. 

\bibitem{Lei-Lin_2011}
Z. Lei, F. Lin, 
Global mild solutions of Navier-Stokes equations,
Comm. Pure Anal. Math. \textbf{64} (2011), 1297 -- 1304.

\bibitem{Majda_2003}
A. Majda,
``Introduction to PDEs and waves for the atmosphere and ocean,''
Courant Lecture Notes in Mathematics \textbf{9}, American Mathematical Society (2003).

\bibitem{Zhang-Yin_2013-P}
Z. Zhang, Z. Yin, Global well-posedness for the generalized Navier-Stokes system,
arXiv:1306.3735.

\end{thebibliography}
\end{document}